\documentclass[a4paper]{amsart}
\usepackage[foot]{amsaddr}

\usepackage[latin1]{inputenc}
\usepackage{amssymb,amsmath,amsthm}
\usepackage{amsfonts}
\usepackage{amsaddr}
\usepackage{enumitem}
\usepackage{ifthen}
\usepackage{tikz}
\usetikzlibrary{matrix}
\usepackage{centernot}
\usepackage{subcaption}

\usepackage{url}
\usepackage{hyphenat}

\newcommand{\IN}{\ensuremath{\mathbb{N}}}

\DeclareMathOperator{\range}{Im}

\DeclareMathOperator{\sI}{\mathbf{I}}
\DeclareMathOperator{\sPsps}{\mathbf{P}^\mathrm{sps}}
\DeclareMathOperator{\sPfsps}{\mathbf{P}_\mathrm{fin}^\mathrm{sps}}
\DeclareMathOperator{\sD}{\mathbf{D}}

\DeclareMathOperator{\sUf}{\mathbf{U}_\mathrm{f}}

\DeclareMathOperator{\sPhf}{\mathbf{P}_\mathrm{hf}}

\DeclareMathOperator{\Mod}{Mod}
\DeclareMathOperator{\Modf}{Mod_{fin}}

\DeclareMathOperator{\qId}{qId}

\DeclareMathOperator{\var}{var}
\DeclareMathOperator{\Id}{Id}
\DeclareMathOperator{\Imp}{Imp}

\DeclareMathOperator{\Hom}{Hom}
\DeclareMathOperator{\Guf}{\mathcal{G}_\mathrm{uf}}

\newcommand{\AG}[1]{\mathbb{A}({#1})}
\newcommand{\reach}[2][]{\ifthenelse{\equal{#1}{}}{[{#2}\rangle}{[{#2}\rangle_{#1}}}
\newcommand{\restr}{\mathord{\upharpoonright}}
\newcommand{\satisfies}{\ensuremath{\models}}
\newcommand{\nsatisfies}{\ensuremath{\centernot\models}}

\theoremstyle{plain}
\newtheorem{theorem}{Theorem}[section]
\newtheorem{proposition}[theorem]{Proposition}
\newtheorem{lemma}[theorem]{Lemma}
\newtheorem{corollary}[theorem]{Corollary}

\newtheoremstyle{clm}
  {}
  {}
  {\itshape}
  {}
  {\rmfamily}
  {\@.}
  { }
  {}
\theoremstyle{clm}
\newtheorem{claim}{Claim}

\theoremstyle{definition}
\newtheorem{definition}[theorem]{Definition}
\newtheorem{example}[theorem]{Example}
\newtheorem{remark}[theorem]{Remark}

\newcommand{\clmqed}[1]{\renewcommand{\qedsymbol}{$\square_\text{Claim~\ref{#1}}$}}

\begin{document}
\title{Graph quasivarieties}

\author{Erkko Lehtonen}

\author{Reinhard P\"oschel}

\address[E. Lehtonen, R. P\"oschel]%
   {Technische Universit\"at Dresden \\
    Institut f\"ur Algebra \\
    01062 Dresden \\
    Germany}

\date{\today}

\begin{abstract}
Introduced by C.~R. Shallon in 1979,
graph algebras establish a useful connection between graph theory and universal algebra.
This makes it possible to investigate graph varieties and graph quasivarieties, i.e., classes of graphs described by identities or quasi\hyp{}identities.
In this paper, graph quasivarieties are characterized as classes of graphs closed under directed unions of isomorphic copies of finite strong pointed subproducts.
\end{abstract}

\maketitle


\section{Introduction}
\label{sec:introduction}

Graph algebras were introduced by C.~R. Shallon in 1979~\cite{Shallon}.
Given a graph $G = (V,E)$
(here always understood as a directed graph without multiple edges, i.e., edges are pairs in $E\subseteq V\times V$),
the corresponding \emph{graph algebra} is $\AG{G} = \langle V \cup \{\infty\}; {\cdot}, \infty \rangle$,
where $\infty \notin V$ is a constant and $\cdot$ is a binary operation given by
$x \cdot y = x$ if $(x,y) \in E$ is a (directed) edge,
otherwise $x \cdot y = \infty$.

Graph algebras establish a useful connection between graph theory and universal algebra.
Graph\hyp{}theoretic notions can be expressed in algebraic terms and,
conversely, algebraic notions, such as identities or quasi\hyp{}identities, can be applied to graphs.
One of the original motivations for graph algebras was to produce classes of examples for
hereditarily non\hyp{}finitely based algebras; see, e.g., \cite{BakMcNWer,McNSha,Shallon}.

Graph algebras never form a variety (in the classical universal algebraic sense).
However, as shown in \cite[Theorems~1.5, 1.6]{KisPosPro}, the investigation of varieties generated by graph algebras and also their subvarieties can be reduced to the investigation of so\hyp{}called \emph{graph varieties}, i.e., classes of graphs definable by identities for their graph algebras.
Such graph varieties were investigated and characterized in \cite{Kiss,Pos-1990};
the corresponding equational logic for graph algebras can be found in \cite{Pos-1989}.
In \cite{KisPosPro,Pos-1989} further results on the lattice of graph varieties are given.

The use of quasi\hyp{}identities (implications) instead of identities leads to the notion of a \emph{graph quasivariety}, i.e., a class of graphs characterizable by quasi\hyp{}identities.
A characterization of graph quasivarieties by closure properties with respect to special graph\hyp{}theoretic constructions (disjoint unions and so\hyp{}called homogeneous subproducts) was provided in \cite{PosWes}, however only for undirected graphs.

In this paper we deal with the remaining open case, namely the characterization of graph quasivarieties for arbitrary (in particular directed) graphs.
The crucial point here in comparision with the case of undirected graphs is that another notion of product (so\hyp{}called strong pointed subproduct) is needed.

Graph quasivarieties are the Galois closed classes of graphs with respect to the Galois connection $\Mod$--$\qId$ induced by the satisfaction relation $\models$ between graphs and quasi\hyp{}identities.
We present a ``Birkhoff-style'' characterization of graph quasivarieties of arbitrary graphs (Theorem~\ref{thm:char}):
\textit{A class of graphs is a graph quasivariety if and only if it is closed under directed unions of isomorphic copies of finite strong pointed subproducts.}

The paper is organized as follows.
In Section~\ref{sec:preliminaries} we introduce necessary notions and notations related to graphs, graph algebras, quasi\hyp{}identities, and quasivarieties.
Section~\ref{sec:sps} mainly deals with the crucial notion of strong pointed subproduct; useful necessary and sufficient conditions for a graph to be isomorphic to a strong pointed subproduct of a family of graphs are derived in Proposition~\ref{prop:IPK-ab}.
The main results are given in Section~\ref{sec:qv}, where
the Characterization Theorem for graph quasivarieties is proved.
As an intermediate step, a version for finite graphs is established first with the help of the conditions of Proposition~\ref{prop:IPK-ab}; this is subsequently extended to arbitrary graphs by taking directed unions.
Finally, in Section~\ref{sec:examples}
we discuss examples and applications of our results.
More specifically, we first particularize our results to the special case of undirected graphs and see how the results of \cite{PosWes} follow.
In order to highlight the difference between varieties and quasivarieties, we provide an example of a graph with the property that the graph variety and the graph quasivariety it generates are distinct.
We also show how certain families of graphs with forbidden subgraphs can be defined by quasi\hyp{}identities.


\section{Preliminaries}
\label{sec:preliminaries}

This paper deals with directed graphs without multiple edges, which will be called simply graphs.

\begin{definition}
A \emph{graph} is an ordered pair $G = (V, E)$ comprising a set $V$ of \emph{vertices} and a set $E \subseteq V \times V$ of \emph{edges.}
We often denote by $V(G)$ and $E(G)$ the set of vertices and the set of edges of a graph $G$, respectively.
An edge of the form $(v,v)$ is called a \emph{loop} at $v$.
A graph is \emph{undirected} if its edge relation is symmetric.
\end{definition}

\begin{definition}
Let $G = (V, E)$ and $G' = (V', E')$ be graphs.
A mapping $h \colon V \to V'$ is a \emph{homomorphism} of $G$ to $G'$, denoted $h \colon G \to G'$, if $h$ maps edges to edges, i.e., for all $u, v \in V$, $(u,v) \in E$ implies $(h(u), h(v)) \in E'$.
A homomorphism $h \colon G \to G'$ is \emph{strong} if, in addition, it maps non\hyp{}edges to non\hyp{}edges, i.e., for all $u, v \in V$, $(u,v) \notin E$ implies $(h(u), h(v)) \notin E'$.
In other words, $h \colon V \to V'$ is a strong homomorphism if for all $u, v \in V$, $(u,v) \in E$ if and only if $(h(u), h(v)) \in E'$.
If $h \colon G \to G'$ is surjective (and strong), then $G'$ is called a \emph{\textup{(}strong\textup{)} homomorphic image} of $G$.
\end{definition}

\begin{definition}
Let $G = (V, E)$ and $G' = (V', E')$ be graphs.
The graph $G$ is a \emph{subgraph} of $G'$, if $V \subseteq V'$ and $E \subseteq E'$, and it is an \emph{induced subgraph}, if additionally $E = E' \cap (V \times V)$.
With no risk of confusion, we may refer to the subgraph of $G'$ induced by a subset $V \subseteq V'$ simply by its set of vertices $V$.
\end{definition}

\begin{definition}
The \emph{direct product} $G = \prod_{i \in I} G_i$ of a family $(G_i)_{i \in I}$ of graphs is defined by
\begin{align*}
V(G) & := \prod_{i \in I} V(G_i) \qquad \text{(Cartesian product),} \\
E(G) & := \bigl\{(a,b) \in V(G) \times V(G) \bigm| \forall i \in I \colon (a(i), b(i)) \in E(G_i) \bigr\}.
\end{align*}
Here, for $a \in V(G)$, $a(i)$ denotes the $i$\hyp{}th component of $a$; thus $a = (a(i))_{i \in I}$.
For each $i \in I$, the \emph{projection} $p_i \colon V(G) \to V(G_i)$, $a \mapsto a(i)$ is a homomorphism $G \to G_i$.
\end{definition}

\begin{definition}
The \emph{union} $G = \bigcup_{i \in I} G_i$ of a family $(G_i)_{i \in I}$ of graphs is defined by
\begin{align*}
V(G) & := \bigcup_{i \in I} V(G_i), &
E(G) & := \bigcup_{i \in I} E(G_i).
\end{align*}
If the vertex sets $V(G_i)$ are pairwise disjoint, then we speak of a \emph{disjoint union} of graphs.
\end{definition}

\begin{definition}
Let $G = (V, E)$ be a graph.
We say that a vertex $v$ is \emph{reachable} from a vertex $u$, if there exists a \emph{walk} from $u$ to $v$ in $G$, i.e., a sequence $v_0, v_1, \dots, v_\ell$ of vertices such that $v_0 = u$, $v_\ell = v$ and $(v_i, v_{i+1}) \in E$ for all $i \in \{0, \dots, \ell - 1\}$.
We denote the set of all vertices of $G$ reachable from $v$ by $\reach[G]{v}$ (or simply by $\reach{v}$).
With no risk of confusion, the same symbol $\reach[G]{v}$ will also be used to refer to the subgraph of $G$ induced by $\reach[G]{v}$.
\end{definition}

\begin{definition}
Let $G = (V, E)$ be a graph, and let $\infty$ be an element distinct from the vertices.
Define a binary operation $\cdot$ on $V \cup \{\infty\}$ by the following rule:
\[
x \cdot y :=
\begin{cases}
x, & \text{if $(x,y) \in E$,} \\
\infty, & \text{otherwise.}
\end{cases}
\]
The algebra $\AG{G} = \langle V \cup \{\infty\}; \cdot, \infty \rangle$ of type $(2,0)$ is called the \emph{graph algebra} of $G$.
\end{definition}

The notion of graph algebra was introduced by Shallon~\cite{Shallon}.
As can easily be seen, all information about a graph $G$ is encoded in an unambiguous way in its graph algebra $\AG{G}$.
Using this encoding, we may view all algebraic properties of $\AG{G}$ as properties of the graph $G$ itself.
In particular, we are interested in the satisfaction of identities or quasi\hyp{}identities by graph algebras.

\begin{definition}
Let $T(X)$ be the set of all \emph{terms} of type $(2,0)$ over a set $X$ of \emph{variables,} using juxtaposition as the binary operation symbol and $\infty$ as the unary operation symbol.
In other words, $T(X)$ is defined inductively as follows:
\begin{enumerate}[label=(\roman*)]
\item\label{terms:1} Every $x \in X$ is a term, and $\infty$ is a term.
\item\label{terms:2} If $t_1$ and $t_2$ are terms, then $(t_1 t_2)$ is a term.
\item $T(X)$ is the set of all terms which can be obtained from \ref{terms:1} and \ref{terms:2} in finitely many steps.
\end{enumerate}

Unless otherwise mentioned, we assume that $X$ is the \emph{standard set of variables} $\{x_1, x_2, x_3, \dots\}$.
The set of variables occurring in a term $t$ is denoted by $\var(t)$.
The leftmost variable occurring in $t$ is denoted by $L(t)$.
A term is called \emph{trivial} if it contains an occurrence of the symbol $\infty$.
\end{definition}

\begin{definition}
An \emph{assignment} of values in a graph algebra $\AG{G}$ for the variables $X$ is a map $h \colon X \to V(G) \cup \{\infty\}$, which will be denoted for brevity by $h \colon X \to \AG{G}$.
Any assignment $h \colon X \to \AG{G}$ extends to a map $\hat{h} \colon T(X) \to V(G) \cup \{\infty\}$ (or briefly $\hat{h} \colon T(X) \to \AG{G}$) as
\[
\hat{h}(t) =
\begin{cases}
h(x), & \text{if $t = x \in X$,} \\
\infty, & \text{if $t = \infty$,} \\
\hat{h}(t_1) \cdot \hat{h}(t_2), & \text{if $t = (t_1 t_2)$.}
\end{cases}
\]
For notational simplicity, we write $h(t)$ for $\hat{h}(t)$, and we call $h(t)$ the \emph{value} of $t$ under the assignment $h$.
The value of $t$ under $h$ obviously depends only on the restriction of $h$ to $\var(t)$;
hence we can define the value of $t$ under a map $h' \colon S \to T$ with $\var(t) \subseteq S \subseteq X$ and $T \subseteq V(G) \cup \{\infty\}$ by taking an arbitrary extension $h \colon X \to \AG{G}$ of $h'$ and setting $h'(t) := h(t)$.
\end{definition}

\begin{definition}
\label{def:term-graph}
To every nontrivial term $t \in T(X)$ we assign a directed graph $G(t) = (V(t), E(t))$, where $V(t) := \var(t)$ and $E(t)$ is defined inductively as
\begin{enumerate}[label=(\roman*)]
\item $E(t) = \emptyset$ if $t = x$ for some $x \in X$,
\item $E(t) = E(t_1) \cup E(t_2) \cup \{ (L(t_1), L(t_2)) \}$ if $t = (t_1 t_2)$.
\end{enumerate}
Note that $G(t)$ is always a connected graph and all its vertices are reachable from $L(t)$.
In fact, a finite graph $G$ is isomorphic to $G(t)$ for some term $t$ if and only if there is a vertex $a \in V(G)$ such that $\reach[G]{a} = V(G)$ (see \cite[Proposition~2.2]{Pos-1990}).
We call such graphs \emph{term graphs.}
Observe that every (strong) homomorphic image of a term graph is a term graph.
\end{definition}

\begin{definition}
Pairs $(t,t')$ in $T(X) \times T(X)$ are called \emph{identities} and they are normally written as $t \approx t'$.
We denote the set of all identities by $\Id(X)$.
A graph $G$ or a graph algebra $\AG{G}$ \emph{satisfies} an identity $t \approx t'$ if for all assignments $h \colon X \to \AG{G}$ we have $h(t) = h(t')$.
We write
\[
G \satisfies t \approx t'
\quad \text{or} \quad
\AG{G} \satisfies t \approx t'
\]
if $\AG{G}$ satisfies $t \approx t'$.
\end{definition}

\begin{proposition}[{see \cite[Proposition~1.5]{PosWes}}]
\label{prop:h(t)=h(L(t))}
Let $G$ be a graph.
\begin{enumerate}[label={\upshape (\arabic*)}]
\item\label{ht=hLt:1}
For a term $t \in T(X)$ and a mapping $h \colon V(t) \to V(G)$, the following are equivalent:
\begin{enumerate}[label={\upshape (\roman*)}]
\item $h(t) \neq \infty$,
\item $h(t) = h(L(t))$,
\item $h$ is a homomorphism of $G(t)$ into $G$.
\end{enumerate}
In particular, if the image of $h$ is not connected, then $h$ is not a homomorphism and we have $h(t) = \infty$.

\item
For nontrivial terms $t, t' \in T(X)$, we have $G \satisfies t \approx t'$ if and only if $\Hom(G(t), G) = \Hom(G(t'), G) =: H$ and $h(L(t)) = h(L(t'))$ for all $h \in H$.
\end{enumerate}
\end{proposition}

\begin{definition}
\label{def:implication}
Pairs $(\alpha, \beta) \in \mathcal{P}_\mathrm{fin}(\Id(X)) \times \Id(X)$ are called \emph{implications} or \emph{quasi\hyp{}identities} (or \emph{universal Horn clauses}; see Wechler~\cite[Section~3.3]{Wechler}) and they are normally written as $\alpha \to \beta$.
The finite set $\alpha$ of identities and the identity $\beta$ are called the \emph{premise} and the \emph{consequence} of the implication, respectively.
If $\alpha = \{s_1 \approx t_1, \dots, s_n \approx t_n\}$ and $\beta = s \approx t$, we can write $\alpha \to \beta$ as $(s_1 \approx t_1 \wedge \dots \wedge s_n \approx t_n) \rightarrow s \approx t$.
We denote the set of all implications by $\Imp(X)$.

A graph $G$ or a graph algebra $\AG{G}$ \emph{satisfies} an implication $\alpha \to \beta$ if for all assignments $h \colon X \to \AG{G}$ we have
\[
\bigl( \forall s_i \approx t_i \in \alpha \colon h(s_i) = h(t_i) \bigr)
\implies
h(s) = h(t).
\]
We write
\[
G \satisfies \alpha \rightarrow \beta
\quad \text{or} \quad
\AG{G} \satisfies \alpha \rightarrow \beta
\]
if $\AG{G}$ satisfies $\alpha \rightarrow \beta$.
Note that every identity can be viewed as an implication, because $G \satisfies t \approx t'$ if and only if $G \satisfies \infty \approx \infty \rightarrow t \approx t'$.
\end{definition}

\begin{definition}
The satisfaction relation between graphs and identities induces a Galois connection in the usual way, and the associated mappings, denoted by $\Mod$ and $\Id$, are defined as follows.
For any set $\mathcal{K}$ of graphs and any set $\Sigma$ of identities, let
\begin{align*}
\Mod \Sigma & := \{ G \in \mathcal{G} \mid \forall t \approx t' \in \Sigma \colon G \satisfies t \approx t' \}, \\
\Id \mathcal{K} & := \{ t \approx t' \in \Id(X) \mid \forall G \in \mathcal{K} \colon G \satisfies t \approx t' \}.
\end{align*}
A set $\mathcal{K}$ of graphs is called a \emph{graph variety} if $\mathcal{K} = \Mod \Sigma$ for some set $\Sigma$ of identities, or, equivalently, if $\mathcal{K} = \Mod \Id \mathcal{K}$.
A set $\Sigma$ of identities is called an \emph{equational theory of graphs} if $\Sigma = \Id \mathcal{K}$ for some set $\mathcal{K}$ of graphs, or, equivalently, if $\Sigma = \Id \Mod \Sigma$.
In other words, the Galois closed sets are the graph varieties and equational theories of graphs.

Similarly, the satisfaction relation between graphs and implications induces a Galois connection given by the mappings $\Mod$ and $\qId$ defined as follows.
For any set $\mathcal{K}$ of graphs and any set $\Sigma$ of implications, let
\begin{align*}
\Mod \Sigma & := \{ G \in \mathcal{G} \mid \forall \alpha \rightarrow \beta \in \Sigma \colon G \satisfies \alpha \rightarrow \beta \}, \\
\qId \mathcal{K} & := \{ \alpha \rightarrow \beta \in \Imp(X) \mid \forall G \in \mathcal{K} \colon G \satisfies \alpha \rightarrow \beta \}.
\end{align*}
A set $\mathcal{K}$ of graphs is called a \emph{graph quasivariety} if $\mathcal{K} = \Mod \Sigma$ for some set $\Sigma$ of implications, or, equivalently, if $\mathcal{K} = \Mod \qId \mathcal{K}$.
A set $\Sigma$ of implications is called an \emph{implicational theory of graphs} if $\Sigma = \qId \mathcal{K}$ for some set $\mathcal{K}$ of graphs, or, equivalently, if $\Sigma = \qId \Mod \Sigma$.
In other words, the Galois closed sets are the graph quasivarieties and implicational theories of graphs.

For a set $\Sigma$ of implications, denote by $\Modf \Sigma$ the set of all finite graphs in $\Mod \Sigma$.
\end{definition}


\section{Strong pointed subproducts and directed unions}
\label{sec:sps}

\begin{definition}
Let $G$ be a graph, and let $\bot$ be an element distinct from the vertices.
The \emph{pointed graph} $G^\bot$ is the graph obtained from $G$ by adding a new vertex $\bot$ and edges from $\bot$ to all vertices, including a loop at $\bot$, that is, 
\begin{align*}
V(G^\bot) & := V(G) \cup \{\bot\}, \\
E(G^\bot) & := E(G) \cup \{(\bot, v) \mid v \in V(G^\bot)\}.
\end{align*}

For a family $(G_i)_{i \in I}$ of graphs, the product $\prod_{i \in I} G_i^\bot$ is called the \emph{pointed product} of $(G_i)_{i \in I}$, and every induced subgraph of $\prod_{i \in I} G_i^\bot$ is called a \emph{pointed subproduct} of $(G_i)_{i \in I}$.
For a vertex $a \in V(\prod_{i \in I} G_i^\bot)$, the \emph{support} of $a$ is the set
\[
Y(a) := \{i \in I \mid a(i) \neq \bot\}.
\]
Note that $b \in \reach{a}$ implies $Y(a) \subseteq Y(b)$.
\end{definition}

\begin{definition}
\label{def:sps}
Let $(G_i)_{i \in I}$ be a family of graphs.
A graph $W$ is a \emph{strong pointed subproduct} of $(G_i)_{i \in I}$ if the following conditions hold:
\begin{enumerate}[label={\upshape(\arabic*)}]
\item\label{def:sps:isg}
$W$ is an induced subgraph of $\prod_{i \in I} G_i^\bot$ (i.e., a pointed subproduct of $(G_i)_{i \in I}$).
\item\label{def:sps:sup}
For all $a \in V(W)$, $Y(a) \neq \emptyset$.
\item\label{def:sps:str}
For all $a \in V(W)$ and for all $k \in Y(a)$, $p_k \restr_{\reach[W]{a}} \colon \reach[W]{a} \to G_k$ is a strong homomorphism.
\end{enumerate}
\end{definition}

\begin{definition}
Let $\mathcal{K}$ be a class of graphs.
\begin{enumerate}[label={\upshape{\arabic*.}}]
\item Denote by $\sI \mathcal{K}$ the class of all graphs isomorphic to members of $\mathcal{K}$.
\item Denote by $\sPsps \mathcal{K}$ the class of all strong pointed subproducts of members of $\mathcal{K}$.
\item Denote by $\sPfsps \mathcal{K}$ the class of all finite members of $\sPsps \mathcal{K}$.
As will be clear from the proof of Proposition~\ref{prop:IPK-ab}, the finite members of $\sPsps \mathcal{K}$ can be obtained as strong pointed subproducts with a finite number of factors.
\end{enumerate}
\end{definition}

\begin{proposition}
\label{prop:IPK-ab}
Let $\mathcal{K}$ be a class of graphs and let $W$ be a graph.
Then $W \in \sI \sPsps \mathcal{K}$ if and only if the following conditions hold:
\begin{enumerate}[label={\upshape(\alph*)}]
\item\label{prop:IPK-ab:a}
For all $a \in V(W)$ there exist a graph $G_a \in \mathcal{K}$ and a strong homomorphism $\varphi_a \colon \reach[W]{a} \to G_a$.
\item\label{prop:IPK-ab:b}
For all $a, a' \in V(W)$ with $a \neq a'$ and $\reach[W]{a} = \reach[W]{a'}$, there exist a graph $G_{\{a,a'\}} \in \mathcal{K}$ and a strong homomorphism $\varphi_{\{a,a'\}} \colon \reach[W]{a} \to G_{\{a,a'\}}$ such that $\varphi_{\{a,a'\}}(a) \neq \varphi_{\{a,a'\}}(a')$.
\end{enumerate}
\end{proposition}

\begin{proof}
``$\Rightarrow$'':
Assume that $W \in \sI \sPsps \mathcal{K}$;
without loss of generality, we may assume that $W \in \sPsps \mathcal{K}$.
Then there is a family $(G_i)_{i \in I}$ of graphs belonging to $\mathcal{K}$ such that conditions \ref{def:sps:isg}, \ref{def:sps:sup}, \ref{def:sps:str} of Definition~\ref{def:sps} hold for $W$.

Let $a \in V(W)$.
By condition \ref{def:sps:sup}, there exists $k \in Y(a)$, and
by condition \ref{def:sps:str}, $p_k \restr_{\reach[W]{a}} \colon \reach[W]{a} \to G_k$ is a strong homomorphism.
Therefore condition \ref{prop:IPK-ab:a} holds with $G_a := G_k$ and $\varphi_a := p_k \restr_{\reach[W]{a}}$.

Let then $a, a' \in V(W)$ with $a \neq a'$ and $\reach[W]{a} = \reach[W]{a'}$.
By condition \ref{def:sps:sup}, $Y(a)$ and $Y(a')$ are nonempty.
Since $a' \in \reach[W]{a}$ implies $Y(a) \subseteq Y(a')$, we must have $Y(a) = Y(a')$.
Consequently, there exists $k \in Y(a)$ such that $a(k) \neq a'(k)$.
By condition \ref{def:sps:str}, $p_k \restr_{\reach[W]{a}} \colon \reach[W]{a} \to G_k$ is a strong homomorphism; moreover $p_k \restr_{\reach[W]{a}}(a) = a(k) \neq a'(k) = p_k \restr_{\reach[W]{a}}(a')$.
Therefore condition \ref{prop:IPK-ab:b} holds with $G_{\{a,a'\}} := G_k$ and $\varphi_{\{a,a'\}} := p_k \restr_{\reach[W]{a}}$.

``$\Leftarrow$'':
Assume that $W$ satisfies conditions \ref{prop:IPK-ab:a} and \ref{prop:IPK-ab:b}.
Let
\[
I := \bigl\{ a \bigm| a \in V(W) \bigr\} \cup \bigl\{ \{a,a'\} \bigm| a, a' \in V(W), \, a \neq a', \, \reach[W]{a} = \reach[W]{a'} \bigr\}.
\]
Let $\widetilde{W}$ be the induced subgraph of $\prod_{i \in I} G_i^\bot$ with
\[
V(\widetilde{W}) :=
\bigl\{ \widetilde{a} \bigm| a \in V(W) \bigr\},
\]
where
$\widetilde{a} := (\widetilde{\varphi}_i(a))_{i \in I}$ (i.e., $p_i(\widetilde{a}) = \widetilde{\varphi}_i(a)$) and
for each $i \in I$ with $i = b$ or $i = \{b,b'\}$, $\widetilde{\varphi}_i \colon W \to G_i^\bot$ is defined by
\[
\widetilde{\varphi}_i(a) :=
\begin{cases}
\varphi_i(a), & \text{if $a \in \reach[W]{b}$,} \\
\bot, & \text{otherwise.}
\end{cases}
\]
Note that $\widetilde{\varphi}_i \colon W \to G_i^\bot$ is a homomorphism; it is not necessarily strong, but $\widetilde{\varphi}_i \restr_{\reach[W]{b}} = \varphi_i$ is a strong homomorphism from $W$ into $G_i$.

First we show that $W \cong \widetilde{W}$.
The mapping $h \colon W \to \widetilde{W}$, $a \mapsto \widetilde{a}$ is a surjective homomorphism by construction.
The claim will follow if we show that $h$ is injective and strong.
For injectivity, let $a, a' \in V(W)$ with $a \neq a'$.
If $\reach[W]{a} = \reach[W]{a'}$, then $\widetilde{\varphi}_{\{a,a'\}}(a) = \varphi_{\{a,a'\}}(a) \neq \varphi_{\{a,a'\}}(a') = \widetilde{\varphi}_{\{a,a'\}}(a')$ by condition~\ref{prop:IPK-ab:b}.
If $\reach[W]{a} \neq \reach[W]{a'}$, then, assuming without loss of generality that $a' \notin \reach[W]{a}$, we have $\widetilde{\varphi}_a(a') = \bot$ and $\widetilde{\varphi}_a(a) = \varphi_a(a) \neq \bot$.
We conclude that there always exists $i \in I$ such that $\widetilde{\varphi}_i(a) \neq \widetilde{\varphi}_i(a')$; therefore, $h(a) = \widetilde{a} \neq \widetilde{a}' = h(a')$.

In order to show that $h$ is strong, assume that $(a,a') \notin E(W)$.
If $a' \notin \reach[W]{a}$, then $\widetilde{a}(a) = \widetilde{\varphi}_a(a) = \varphi_a(a) \neq \bot$ and $\widetilde{a}'(a) = \widetilde{\varphi}_a(a') = \varphi_a(a') = \bot$, so $(\widetilde{a}, \widetilde{a}') \notin E(\widetilde{W})$.
If $a' \in \reach[W]{a}$, then since $\varphi_a \restr_{\reach[W]{a}} \colon \reach[W]{a} \to G_a$ is a strong homomorphism by condition \ref{prop:IPK-ab:a}, we have $(\widetilde{\varphi}_a(a), \widetilde{\varphi}(a')) = (\varphi_a(a), \varphi_a(a')) \notin E(G_a)$, so $(\widetilde{a}, \widetilde{a}') \notin E(\widetilde{W})$.

It remains to show that $\widetilde{W}$ is a strong pointed subproduct of $(G_i)_{i \in I}$, that is, $\widetilde{W}$ satisfies the conditions of Definition~\ref{def:sps}.
Condition \ref{def:sps:isg} holds by definition.
Condition \ref{def:sps:sup} is clear, because $\widetilde{a}(a) = \widetilde{\varphi}_a(a) = \varphi_a(a) \neq \bot$ for all $a \in V(W)$.
As for condition \ref{def:sps:str}, let $\widetilde{a} \in V(\widetilde{W})$ and $k \in Y(\widetilde{a})$.
Then $k = b$ or $k = \{b,b'\}$ for some $b, b' \in V(W)$, and $\bot \neq \widetilde{a}(k) = \widetilde{\varphi}_k(a)$.
By the definition of $\widetilde{\varphi}_k$, this implies $\widetilde{\varphi}_k(a) = \varphi_k(a)$ and $a \in \reach[W]{b}$; hence $\reach[W]{a} \subseteq \reach[W]{b}$.
By condition \ref{prop:IPK-ab:a} or \ref{prop:IPK-ab:b}, $\varphi_k \colon \reach[W]{b} \to G_k$ is a strong homomorphism, and therefore also its restriction $\varphi_k \restr_{\reach[W]{a}} \colon \reach[W]{a} \to G_k$ is strong.
The isomorphism $h \colon W \to \widetilde{W}$ is strong, so the composition
$\varphi_k \restr_{\reach[W]{a}} \circ h^{-1} \restr_{\reach[\widetilde{W}]{\widetilde{a}}}$
is also strong and, in fact, equals $p_k \restr_{\reach[\widetilde{W}]{\widetilde{a}}}$, because
$p_k(\widetilde{x}) = \widetilde{\varphi}_k(x) = \varphi_k(x) = \varphi_k(h^{-1}(\widetilde{x}))$ for all $x \in \reach[W]{a} \subseteq \reach[W]{b}$.
\end{proof}

\begin{remark}
\label{rem:IPsps-closures}
It is not difficult to verify with the help of Definition~\ref{def:sps} or Proposition~\ref{prop:IPK-ab} that
for any class $\mathcal{K}$ of graphs,
$\sI \sPsps \mathcal{K}$ contains
induced subgraphs,
disjoint unions,
direct products,
and
strong homomorphic images
of members of $\mathcal{K}$.
\end{remark}

\begin{definition}
A family $(G_i)_{i \in I}$ of graphs is \emph{directed} if for all $i, j \in I$ there exists $k \in I$ such that both $G_i$ and $G_j$ are induced subgraphs of $G_k$.
The union of a directed family of graphs is called a \emph{directed union.}
For a class $\mathcal{K}$ of graphs, denote by $\sD \mathcal{K}$ the class of all directed unions of members of $\mathcal{K}$.
\end{definition}

\begin{remark}
Note that each $G_i$ is an induced subgraph of the directed union $\bigcup_{i \in I} G_i$.
Moreover, every graph is a directed union of its finite induced subgraphs.
\end{remark}


\section{Characterization of graph quasivarieties}
\label{sec:qv}

We are now ready to characterize graph quasivarieties in terms of closure operators built using purely graph\hyp{}theoretical constructions.

\begin{theorem}[Characterization Theorem]
\label{thm:char}
For any class $\mathcal{K}$ of graphs,
\[
\Mod \qId \mathcal{K} = \sD \sI \sPfsps \mathcal{K}.
\]
\end{theorem}

The remainder of this section will lead to the proof of the Characterization Theorem.
We first use Proposition~\ref{prop:IPK-ab} to describe the finite members of a graph quasivariety $\Mod \qId \mathcal{K}$.
As an intermediate step, we obtain a ``finite version'' of the Characterization Theorem (see Theorem~\ref{thm:char-finite}), which asserts that the finite members of $\Mod \qId \mathcal{K}$ are precisely the isomorphic copies of finite strong pointed subproducts of members of $\mathcal{K}$.
Finally, we show that the general members of $\Mod \qId \mathcal{K}$ are precisely the directed unions of its finite members.

We begin with a straightforward way of encoding a graph as a set of identities.
This construction will be used repeatedly in the sequel.

\begin{definition}
\label{def:Sigma-Gamma}
Let $G$ be a graph.
For each $a \in V(G)$, let $x_a := a$, and let $X_{V(G)} := \{x_a \mid a \in V(G)\} = V(G)$.
Define the sets $\Gamma_\mathrm{e}(G)$, $\Gamma_\mathrm{n}(G)$, and $\Sigma(G)$ of identities over the set $X_{V(G)}$ of variables as follows:
\begin{align*}
\Gamma_\mathrm{e}(G) &:= \{ x_a x_b \approx x_a \mid a, b \in V(G), \, (a,b) \in E(G) \}, \\
\Gamma_\mathrm{n}(G) &:= \{ x_a x_b \approx \infty \mid a, b \in V(G), \, (a,b) \notin E(G) \}, \\
\Sigma(G) &:= \Gamma_\mathrm{e}(G) \cup \Gamma_\mathrm{n}(G).
\end{align*}
The vertices of $G$ play a double role here, and we have intentionally created two different ways of referring to them: as $a$ when viewed as vertices of $G$ and as $x_a$ when viewed as variables.
\end{definition}

\begin{lemma}
\label{lem:SigmaG}
Let $G$ and $H$ be graphs.
\begin{enumerate}[label={\upshape(\roman*)}]
\item\label{lem:SigmaG-Gamma} An assignment $h \colon X_{V(G)} \to V(H)$ makes every identity of $\Gamma_\mathrm{e}(G)$ true in $\AG{H}$ if and only if $h$ is a homomorphism of $G$ into $H$.
\item\label{lem:SigmaG-Sigma} An assignment $h \colon X_{V(G)} \to V(H)$ makes every identity of $\Sigma(G)$ true in $\AG{H}$ if and only if $h$ is a strong homomorphism of $G$ into $H$.
\item\label{lem:SigmaG-inf} If $h \colon X_{V(G)} \to V(H) \cup \{\infty\}$ is an assignment that makes every identity of $\Gamma_\mathrm{e}(G)$ true in $\AG{H}$ and $h(x_a) \neq \infty$, then $h(x_b) \neq \infty$ for all $b \in \reach[G]{a}$.
\end{enumerate}
\end{lemma}

\begin{proof}
\ref{lem:SigmaG-Gamma} and \ref{lem:SigmaG-Sigma}:
Assume that $h \colon X_{V(G)} \to V(H)$ makes every identity of $\Gamma_\mathrm{e}(G)$ true in $\AG{H}$.
Then for all $a, b \in V(G)$ with $(a,b) \in E(G)$ we have $h(a) h(b) = h(x_a) h(x_b) = h(x_a) = h(a)$ in $\AG{H}$, that is, $(h(a),h(b)) \in E(H)$.
Hence, $h$ is a homomorphism of $G$ into $H$.

If, additionally, $h$ makes every identity of $\Gamma_\mathrm{n}(G)$ true in $\AG{H}$, then
for all $a, b \in V(G)$ with $(a,b) \notin E(G)$ we have $h(a) h(b) = h(x_a) h(x_b) = \infty$ in $\AG{H}$, that is, $(h(a),h(b)) \notin E(H)$.
This implies that $h$ is a strong homomorphism.

Assume now that $h$ is a homomorphism of $G$ into $H$.
Then for all $a, b \in V(G)$ with $(a,b) \in E(G)$ we have $(h(a),h(b)) \in E(H)$, whence $h(x_a) h(x_b) = h(a) h(b) = h(a) = h(x_a)$.
Therefore $h$ makes every identity of $\Gamma_\mathrm{e}(G)$ true in $\AG{H}$.

If, additionally, $h$ is a strong homomorphism, then
for all $a, b \in V(G)$ with $(a,b) \notin E(G)$ we have $(h(a),h(b)) \notin E(H)$, so $h(x_a) h(x_b) = h(a) h(b) = \infty$.
Consequently, $h$ makes every identity of $\Gamma_\mathrm{n}(G)$ and hence every identity of $\Sigma(G)$ true in $\AG{H}$.

\ref{lem:SigmaG-inf}:
Suppose, to the contrary, that there is $b \in \reach[W]{a}$ such that $h(x_b) = \infty$.
Then there exist vertices $c_1, c_2, \dots, c_n \in \reach[G]{a}$ such that $(a,c_1), (c_1,c_2), \dots, (c_n,b) \in E(G)$;
hence $\Gamma_\mathrm{e}(G)$ contains the identities $x_a x_{c_1} \approx x_a$, $x_{c_1} x_{c_2} \approx x_{c_1}$, \dots, $x_{c_n} x_b \approx x_{c_n}$.
Since the assignment $h$ makes every identity of $\Gamma_\mathrm{e}(G)$ true in $\AG{H}$, it also makes true the identity
$x_a (x_{c_1} (x_{c_2} ( \cdots ( x_{c_{n-1}} ( x_{c_n} x_b ) ) \cdots ))) \approx x_a$ that is obtained from the above identities by the replacement rule.
Since $h(x_b) = \infty$, the value of the left side of this identity under $h$ is $\infty$, while the value of the right side is $h(x_a) \neq \infty$,
a contradiction.
\end{proof}

\begin{proposition}
\label{prop:Mod-qId-ab}
Let $\mathcal{K}$ be a class of graphs and let $W$ be a finite graph.
Then $W \in \Modf \qId \mathcal{K}$ if and only if
conditions \ref{prop:IPK-ab:a} and \ref{prop:IPK-ab:b} of Proposition~\ref{prop:IPK-ab} hold for $W$.
\end{proposition}

\begin{proof}
``$\Rightarrow$'':
We prove the contrapositive.
Let $W$ be a graph, and assume that condition \ref{prop:IPK-ab:a} or \ref{prop:IPK-ab:b} does not hold for $W$.
Our goal is to construct an implication $\alpha \rightarrow \beta$ such that $G \satisfies \alpha \rightarrow \beta$ for all $G \in \mathcal{K}$ but $W \nsatisfies \alpha \rightarrow \beta$.
This means that $W \notin \Mod \qId \mathcal{K}$.

Assume first that condition \ref{prop:IPK-ab:a} does not hold for $W$.
Then there is a vertex $a \in V(W)$ such that there exists no strong homomorphism of $\reach[W]{a}$ into any graph $G \in \mathcal{K}$.
Consider the implication $\alpha \rightarrow \beta$ with $\alpha := \Sigma(\reach[W]{a})$, $\beta := x_a \approx \infty$ over the set $X_{\reach[W]{a}} = \{x_b \mid b \in \reach[W]{a}\} = \reach[W]{a}$ of variables.
(Working with an implication over a non\hyp{}standard set of variables is of no important consequence here. Since $X_{\reach[W]{a}}$ is a finite set, we can rename its elements to get an equivalent implication over the standard set of variables.)
We clearly have $W \nsatisfies \alpha \rightarrow \beta$,
because $h \colon X_{\reach[W]{a}} \to V(W) \cup \{\infty\}$, $x_b \mapsto b$, is an assignment that makes $\alpha$ true and $\beta$ false.

On the other hand, $G \satisfies \alpha \rightarrow \beta$ for all $G \in \mathcal{K}$.
For, suppose, to the contrary, that there is a graph $G' \in \mathcal{K}$ such that $G' \nsatisfies \alpha \rightarrow \beta$.
Then there exists an assignment $h \colon X_{\reach[W]{a}} \to V(G') \cup \{\infty\}$
that makes every identity in $\alpha$ true and $\beta$ false,
that is,
\begin{equation}
\begin{aligned}
h(x_b) h(x_{b'}) &= h(x_b) && \text{for all $b, b' \in \reach[W]{a}$ with $(b,b') \in E(W)$  and} \\
h(x_b) h(x_{b'}) &= \infty && \text{for all $b, b' \in \reach[W]{a}$ with $(b,b') \notin E(W)$,}
\end{aligned}
\label{eq:hbhb'}
\end{equation}
but
$h(x_a) \neq \infty$.
By Lemma~\ref{lem:SigmaG}\ref{lem:SigmaG-inf}, we have $h(x_b) \neq \infty$, i.e., $h(x_b) \in V(G')$ for all $b \in \reach[W]{a}$.
It follows from Lemma~\ref{lem:SigmaG}\ref{lem:SigmaG-Sigma} that $h$ is a strong homomorphism of $\reach[W]{a}$ into $G'$.
This contradicts our assumption that no strong homomorphism exists from $\reach[W]{a}$ into any graph in $\mathcal{K}$.

Assume now that condition \ref{prop:IPK-ab:b} does not hold for $W$.
Then there exist $a, a' \in V(W)$ with $a \neq a'$, $\reach[W]{a} = \reach[W]{a'}$ such that for every graph $G \in \mathcal{K}$ and for every strong homomorphism $\varphi \colon \reach[W]{a} \to G$ we have $\varphi(a) = \varphi(a')$.
Consider the implication $\alpha \rightarrow \beta$ with $\alpha := \Sigma(\reach[W]{a})$, $\beta := x_a \approx x_{a'}$ over the set $X_{\reach[W]{a}}$ of variables.
We clearly have $W \nsatisfies \alpha \rightarrow \beta$,
because $h \colon X_{\reach[W]{a}} \to V(W) \cup \{\infty\}$, $x_b \mapsto b$, is an assignment that makes $\alpha$ true and $\beta$ false.

On the other hand, $G \satisfies \alpha \rightarrow \beta$ for all $G \in \mathcal{K}$.
For, suppose, to the contrary, that there is a graph $G' \in \mathcal{K}$ such that $G' \nsatisfies \alpha \rightarrow \beta$.
Then there is an assignment $h \colon X_{\reach[W]{a}} \to V(G') \cup \{\infty\}$
that makes every identity in $\alpha$ true and $\beta$ false; hence $h(x_a) \neq h(x_{a'})$.
Without loss of generality, we may assume that $h(x_a) \neq \infty$.
As shown above, this implies that $h(x_b) \neq \infty$, i.e., $h(x_b) \in V(G')$ for all $b \in \reach[W]{a}$.
Consequently, $h$ is a strong homomorphism of $\reach[W]{a}$ into $G'$ by Lemma~\ref{lem:SigmaG}\ref{lem:SigmaG-Sigma};
moreover $h(a) = h(x_a) \neq h(x_{a'}) = h(a')$.
This is a contradiction to our assumption on the strong homomorphisms into graphs in $\mathcal{K}$ (take $G = G'$ and $\varphi = h$).

``$\Leftarrow$'':
Assume that conditions \ref{prop:IPK-ab:a} and \ref{prop:IPK-ab:b} of Proposition~\ref{prop:IPK-ab} hold for $W$.
Let $\alpha \rightarrow \beta \in \qId \mathcal{K}$, with
$\alpha = \{ s_1 \approx t_1, \dots, s_n \approx t_n \}$,
$\beta = s \approx t$.
We need to show that $W \satisfies \alpha \rightarrow \beta$.

Suppose, to the contrary, that $W \nsatisfies \alpha \rightarrow \beta$.
Then there must exist an assignment $h \colon X \to V(W) \cup \{\infty\}$ such that
$h(s_1) = h(t_1)$, \dots, $h(s_n) = h(t_n)$
but $h(s) \neq h(t)$.
Without loss of generality, we may assume that $h(s) \neq \infty$; thus $h(s) = h(L(s))$.

Let
\[
I := \bigl\{ a \bigm| a \in V(W) \bigr\} \cup \bigl\{ \{a,a'\} \bigm| a, a' \in V(W), \, a \neq a', \, \reach[W]{a} = \reach[W]{a'} \bigr\}.
\]
For each $i \in I$, with $i = a$ or $i = \{a,a'\}$, let $G_i$ and $\varphi_i \colon \reach[W]{a} \to G_i$ be the graph and the strong homomorphism provided by condition \ref{prop:IPK-ab:a} or \ref{prop:IPK-ab:b}, and define the map $h_i \colon X \to V(G_i) \cup \{\infty\}$ by the rule
\[
h_i(x) :=
\begin{cases}
\varphi_i(h(x)), & \text{if $h(x) \in \reach[W]{a}$,} \\
\infty, & \text{otherwise.}
\end{cases}
\]
We establish some useful properties of the maps $h_i$ in the following three claims.

\begin{claim}
\label{clm:1}
Let $i \in I$ with $i = a$ or $i = \{a,a'\}$, and let $t' \in T(X)$.
If $h(t') \notin \reach[W]{a}$, then $h_i(t') = \infty$.
\end{claim}

\begin{proof}
Assume $h(t') \notin \reach[W]{a}$.
If there is an $x \in \var(t')$ with $h(x) \notin \reach[W]{a}$, then $h_i(x) = \infty$; hence $h_i(t') = \infty$.
If for all $x \in \var(t')$ we have $h(x) \in \reach[W]{a}$, then $h \restr_{\var(t')}$ is not a homomorphism $G(t') \to W$ (otherwise $h(t') = h(L(t')) \in \reach[W]{a}$, a contradiction; see Proposition~\ref{prop:h(t)=h(L(t))}~\ref{ht=hLt:1}).
Therefore there exists an edge $(x,y) \in E(G(t'))$ such that $(h(x),h(y)) \notin E(W)$.
Since $\varphi_i$ is a strong homomorphism, we have $(h_i(x),h_i(y)) = (\varphi_i(h(x)),\varphi_i(h(y))) \notin E(G_i)$; hence $h_i(t') = \infty$.
\clmqed{clm:1}
\end{proof}

\begin{claim}
\label{clm:2}
Let $i \in I$ with $i = a$ or $i = \{a,a'\}$, and let $t' \in T(X)$.
If $h(t') = h(L(t')) =: b \in \reach[W]{a}$, then $h_i(t') = h_i(L(t')) = \varphi_i(b)$.
\end{claim}

\begin{proof}
Assume $h(t') = h(L(t')) =: b \in \reach[W]{a}$.
Then $b \neq \infty$, and $h(x) \in \reach[W]{a} \subseteq \reach[W]{a}$ for all $x \in \var(t')$, because $x \in \reach[G(t')]{L(t')}$ and $h \restr_{\var(t')}$ is a homomorphism $G(t') \to W$.
Therefore $h_i \restr_{\var(t')} = \varphi_i \circ h \restr_{\var(t')}$ is a homomorphism $G(t') \to G_i$,
so Proposition~\ref{prop:h(t)=h(L(t))} yields $h_i(t') = h_i(L(t')) = \varphi_i(h(L(t'))) = \varphi_i(b)$.
\clmqed{clm:2}
\end{proof}

\begin{claim}
\label{clm:3}
Let $i \in I$ with $i = a$ or $i = \{a,a'\}$, and let $s', t' \in T(X)$.
If $h(s') = h(t')$, then $h_i(s') = h_i(t')$.
\end{claim}

\begin{proof}
Let $b := h(s') = h(t')$.
If $b \notin \reach[W]{a}$, then $h_i(s') = \infty = h_i(t')$ by Claim~\ref{clm:1}.
If $b \in \reach[W]{a}$, then $h_i(s') = h_i(L(s')) = \varphi_i(b) = h_i(L(t')) = h_i(t')$ by Claim~\ref{clm:2}.
\clmqed{clm:3}
\end{proof}

\noindent
Now we return to the proof of Proposition~\ref{prop:Mod-qId-ab}.
Let $a := h(s) = h(L(s))$.
We need to distinguish between the cases $h(t) = \infty$ and $h(t) \neq \infty$.

Case 1: $h(t) = \infty$.
We have $h_a(s_j) = h_a(t_j)$ for all
$j \in \{1, \dots, n\}$
by Claim~\ref{clm:3},
$h_a(s) = h_a(L(s)) = \varphi_a(a) \neq \infty$ by Claim~\ref{clm:2}, and
$h_a(t) = \infty$ by Claim~\ref{clm:1}.
Thus $G_a \nsatisfies \alpha \rightarrow \beta$, as witnessed by the assignment $h_a$.
We have reached a contradiction.

Case 2: $h(t) \neq \infty$.
Let $a' := h(t) = h(L(t))$.
We need to consider three subcases.

Case 2.1: $a' \notin \reach[W]{a}$.
We have $h_a(s_j) = h_a(t_j)$ for all
$j \in \{1, \dots, n\}$
by Claim~\ref{clm:3},
$h_a(s) = h_a(L(s)) = \varphi_a(a) \neq \infty$ by Claim~\ref{clm:2}, and
$h_a(t) = \infty$ by Claim~\ref{clm:1}.
Thus $G_a \nsatisfies \alpha \rightarrow \beta$, a contradiction.

Case 2.2: $a \notin \reach[W]{a'}$.
Analogously to the previous case, we obtain $G_{a'} \nsatisfies \alpha \rightarrow \beta$, a contradiction.

Case 2.3: $\reach[W]{a} = \reach[W]{a'}$.
We have $h_{\{a,a'\}}(s_j) = h_{\{a,a'\}}(t_j)$ for all
$j \in \{1, \dots, n\}$
by Claim~\ref{clm:3},
and
$h_{\{a,a'\}}(s) = \varphi_{\{a,a'\}}(a) \neq \varphi_{\{a,a'\}}(a') = h_{\{a,a'\}}(t)$ by Claim~\ref{clm:2} and condition \ref{prop:IPK-ab:b}.
Thus $G_{\{a,a'\}} \nsatisfies \alpha \rightarrow \beta$, a contradiction.
\end{proof}

\begin{theorem}
\label{thm:char-finite}
For any class $\mathcal{K}$ of graphs,
$\Modf \qId \mathcal{K} = \sI \sPfsps \mathcal{K}$.
\end{theorem}

\begin{proof}
Follows immediately from Propositions~\ref{prop:IPK-ab} and \ref{prop:Mod-qId-ab}.
\end{proof}

\begin{remark}
Proposition~\ref{prop:Mod-qId-ab} does not hold for infinite graphs, as will be illustrated by Example~\ref{ex:paths}.
Nonetheless, the assumption on the finiteness of $W$ can be relaxed a bit.
The proof requires finiteness only at the point where we construct implications with premise $\Sigma(\reach[W]{a})$; this must be a finite set of identities.
Consequently, Proposition~\ref{prop:Mod-qId-ab} holds for every graph $W$ in which the set $\reach[W]{a}$ is finite for every vertex $a \in V(W)$.
\end{remark}

In order to deal with arbitrary members of $\Mod \qId \mathcal{K}$, we need to consider directed unions.

\begin{lemma}
\label{lem:ModSigma-dirun}
For any set $\Sigma$ of implications, $\Mod \Sigma$ is closed under directed unions.
\end{lemma}

\begin{proof}
Let $(G_i)_{i \in I}$ be a directed family of members of $\Mod \Sigma$, and let $G$ be the directed union of $(G_i)_{i \in I}$.
Let $\alpha \to \beta \in \Sigma$, with $\alpha = \{s_1 \approx t_1, \dots, s_n \approx t_n\}$, $\beta = s \approx t$.
Let $Y$ be the set of variables occurring in $\alpha \to \beta$, and let $h \colon Y \to V(G) \cup \{\infty\}$.
Since the set $Y$ is finite, the set $H := (\range h) \setminus \{\infty\}$ is finite as well.
Since $(G_i)_{i \in I}$ is directed, there exists $k \in I$ such that $H \subseteq V(G_k)$.
By our assumption, $G_k \satisfies \alpha \to \beta$, so the following implication holds:
\[
\bigl( h(s_1) = h(t_1) \wedge \dots \wedge h(s_n) = h(t_n) \bigr)
\implies
h(s) = h(t).
\]
Since $h$ was chosen arbitrarily, we have $G \satisfies \alpha \rightarrow \beta$; consequently, $G \in \Mod \Sigma$.
\end{proof}

\begin{corollary}
\label{cor:W-ModSigma-fin}
For any graph $W$ and for any set $\Sigma$ of implications, we have $W \in \Mod \Sigma$ if and only if $W' \in \Modf \Sigma$ for every finite induced subgraph $W'$ of $W$.
\end{corollary}

\begin{proof}
``$\Rightarrow$'':
Clear, because $\Mod \Sigma$ is closed under taking induced subgraphs.

``$\Leftarrow$'':
Assume that $W' \in \Modf \Sigma$ for all finite induced subgraphs $W'$ of $W$.
Since $W$ is a directed union of all its finite induced subgraphs, Lemma~\ref{lem:ModSigma-dirun} implies $W \in \Mod \Sigma$.
\end{proof}

\begin{proof}[Proof of Theorem~\ref{thm:char}.]
Let $W \in \Mod \qId \mathcal{K}$.
By Corollary~\ref{cor:W-ModSigma-fin} and Theorem~\ref{thm:char-finite}, $W' \in \Modf \qId \mathcal{K} = \sI \sPfsps \mathcal{K}$ for every finite induced subgraph $W'$ of $W$.
Since $W$ is a directed union of all its finite induced subgraphs, we obtain $W \in \sD \sI \sPfsps \mathcal{K}$.

Let then $W \in \sD \sI \sPfsps \mathcal{K}$.
Then $W$ is a directed union of a family $(W_i)_{i \in I}$ of graphs in $\sI \sPfsps \mathcal{K}$.
By Theorem~\ref{thm:char-finite}, for all $i \in I$ we have $W_i \in \Modf \qId \mathcal{K} \subseteq \Mod \qId \mathcal{K}$.
Lemma~\ref{lem:ModSigma-dirun} implies $W \in \Mod \qId \mathcal{K}$.
\end{proof}


\section{Examples and applications}
\label{sec:examples}

\subsection*{Quasivarieties of undirected graphs}

As already mentioned in the introduction, 
our work extends the results for quasivarieties of undirected graphs that were reported in~\cite{PosWes}.
These results (in particular \cite[Theorem~2.8]{PosWes}) can be obtained as a special case of our current results.
Indeed, if $G$ is an undirected graph, then for every vertex $a$, the subgraph induced by the set $\reach[G]{a}$ of vertices reachable from $a$ is exactly the connected component of $G$ containing the vertex $a$.
Therefore, in the case of finite undirected graphs, Proposition~\ref{prop:IPK-ab} and Theorem~\ref{thm:char} lead to the following proposition (here $\Guf$ denotes the class of all undirected finite graphs).

\begin{proposition}
\label{prop:fu}
Let $\mathcal{K} \subseteq \Guf$ and $W \in \Guf$.
Then $W \in \Mod \qId \mathcal{K}$ if and only if the following conditions hold:
\begin{enumerate}[label={\upshape(\arabic*)}]
\item\label{prop:fu:a}
For every one\hyp{}element connected component $H$ of $W$, there exist a graph $G_H \in \mathcal{K}$ and a strong homomorphism $\varphi_H \colon H \to G_H$.
\item\label{prop:fu:b}
For any pair $a$, $a'$ of distinct vertices belonging to the same connected component $H$ of $W$, there exist a graph $G_{\{a,a'\}} \in \mathcal{K}$ and a strong homomorphism $\varphi_{\{a,a'\}} \colon H \to G_{\{a,a'\}}$ such that $\varphi_{\{a,a'\}}(a) \neq \varphi_{\{a,a'\}}(a')$.
\end{enumerate}
\end{proposition}

Note that \ref{prop:fu:a} and \ref{prop:fu:b} ensure condition \ref{prop:IPK-ab:a} of Proposition~\ref{prop:IPK-ab} for all $a \in V(W)$.
Condition \ref{prop:fu:a} simply means: if $W$ has an isolated vertex, with or without a loop, then at least one graph in $\mathcal{K}$ must have a vertex with or without a loop, respectively.

Proposition~\ref{prop:fu} coincides with the Characterization Theorem 2.8 of \cite{PosWes} (which states that $\Modf \qId \mathcal{K} = \sI \sUf \sPhf \mathcal{K}$, where the operators $\sUf$ and $\sPhf$ provide the finite disjoint unions and the finite so\hyp{}called homogeneous subproducts, respectively);
however, in Proposition~\ref{prop:fu} we avoided the notion of homogeneous subproduct (see \cite[Definition~2.3]{PosWes}) and used their equivalent characterization by the properties \ref{prop:fu:a} and \ref{prop:fu:b}, as given in \cite[Proposition~2.4]{PosWes}.

\subsection*{Graph varieties vs.\ graph quasivarieties}

Since every identity can be viewed as an implication, every graph variety is a graph quasivariety, and it holds for any class $\mathcal{K}$ of graphs that $\Mod \qId \mathcal{K} \subseteq \Mod \Id \mathcal{K}$.
This inclusion may be proper, as shown by the following example.

\begin{example}
Let $G_0$ be the graph shown in Figure~\ref{fig:G0}, i.e., $V(G_0) = \{0,1\}$, $E(G_0) = \{(0,1), (1,0), (1,1)\}$.
As shown by McNulty and Shallon \cite[Example~3]{McNSha}, $\Mod \Id G_0$ contains every finite undirected graph without loops; in particular, $K_3$, the complete loopless graph on three vertices (see Figure~\ref{fig:K3}), belongs to $\Mod \Id G_0$.
However, as is easy to see, there is no strong homomorphism of $K_3$ into $G_0$;
hence, condition \ref{prop:IPK-ab:a} of Proposition~\ref{prop:IPK-ab} does not hold for $W = K_3$ and $\mathcal{K} = \{G_0\}$ (note that $\reach[K_3]{a} = K_3$ for every vertex $a$ of $K_3$).
Consequently, $K_3 \notin \Mod \qId G_0$ by Proposition~\ref{prop:Mod-qId-ab}.
This fact is also witnessed by the implication
\[
( x(yz) \approx x  \wedge  zx \approx \infty ) \rightarrow x(yy) \approx xy,
\]
which is satisfied by $G_0$ but not by $K_3$.
\end{example}

\begin{figure}
\tikzset{every node/.style={circle,draw,inner sep=1.5,fill=black}, every path/.style={thick}}
\begin{minipage}[b]{0.3\textwidth}
\begin{center}
\begin{tikzpicture}
\node (0) at (0,0) [label=below:{$0$}]{};
\node (1) at (1,0) [label=below:{$1$}]{};
\draw (0) edge (1);
\draw (1) edge[out=125,in=55,looseness=18] (1);
\end{tikzpicture}
\end{center}
\caption{$G_0$}
\label{fig:G0}
\end{minipage}
\begin{minipage}[b]{0.3\textwidth}
\begin{center}
\begin{tikzpicture}
\node (0) at (90:0.66) [label=above:{$0$}]{};
\node (1) at (210:0.66) [label=below:{$1$}]{};
\node (2) at (330:0.66) [label=below:{$2$}]{};
\draw (0) -- (1) -- (2) -- (0);
\end{tikzpicture}
\end{center}
\caption{$K_3$}
\label{fig:K3}
\end{minipage}
\begin{minipage}[b]{0.3\textwidth}
\begin{center}
\begin{tikzpicture}
\node (0) at (0,0) [label=below:{$0$}]{};
\node (1) at (1,0) [label=below:{$1$}]{};
\draw (0) edge (1);
\draw (0) edge[out=125,in=55,looseness=18] (0);
\draw (1) edge[out=125,in=55,looseness=18] (1);
\end{tikzpicture}
\end{center}
\caption{$K'_2$}
\label{fig:K'2}
\end{minipage}
\end{figure}

\begin{example}
With the help of Propositions~\ref{prop:IPK-ab} and \ref{prop:Mod-qId-ab}, we see also that $K'_2$, the complete graph with loops on two vertices (see Figure~\ref{fig:K'2}), does not belong to $\Mod \qId G_0$.
The only strong homomorphism of $K'_2$ into $G_0$ is the constant map sending both vertices to $1$.
However, $\reach[K'_2]{0} = \reach[K'_2]{1} = K'_2$, so condition \ref{prop:IPK-ab:b} of Proposition~\ref{prop:IPK-ab} does not hold for $W = K'_2$ and $\mathcal{K} = \{G_0\}$.
Consequently, $K'_2 \notin \Mod \qId G_0$ by Proposition~\ref{prop:Mod-qId-ab}.
This fact is also witnessed by the implication
\[
( x(yz) \approx x  \wedge  xx \approx x  \wedge zz \approx z ) \rightarrow x(yz) \approx z(yx),
\]
which is satisfied by $G_0$ but not by $K'_2$.
\end{example}

In fact, the above examples generalize easily to the following description of the quasivariety generated by $G_0$.

\begin{proposition}
A graph $G$ belongs to $\Mod \qId G_0$ if and only if 
it is undirected and each connected component is isomorphic to an induced subgraph of $G_0$ (i.e., either $G_0$, a loopless isolated vertex, or an isolated vertex with a loop).
\end{proposition}

\begin{proof}
``$\Rightarrow$'':
Assume that $G \in \Mod \qId G_0$.
Since $G_0$ is undirected and undirectedness is an implicational (in fact, equational) property, it follows that also $G$ must be undirected.
Let $H$ be a finite induced subgraph of $G$.
By Corollary~\ref{cor:W-ModSigma-fin}, $H \in \Mod \qId G_0$; hence, by Proposition~\ref{prop:Mod-qId-ab}, $H$ satisfies conditions \ref{prop:IPK-ab:a} and \ref{prop:IPK-ab:b} of Proposition~\ref{prop:IPK-ab}.

Let now $H'$ be a connected component of $H$, and let $a, a' \in V(H')$.
Since $G$ is undirected, we have $\reach[H]{a} = \reach[H]{a'} = H'$.
If $a \neq a'$, then by condition \ref{prop:IPK-ab:b} there exists a strong homomorphism $\varphi_{\{a,a'\}} \colon H' \to G_0$ such that $\varphi_{\{a,a'\}}(a) \neq \varphi_{\{a,a'\}}(a')$.
Observe that any strong homomorphism into $G_0$ must map loopless vertices to $0$ and vertices with a loop to $1$.
Consequently, one of the vertices $a$ and $a'$ is loopless and the other has a loop.
Since the choice of $a$ and $a'$ was arbitrary, it follows that $H'$ has no more than one loopless vertex and no more than one vertex with a loop; altogether $H'$ has at most two vertices.
Therefore, $H'$ is isomorphic to an induced subgraph of $G_0$.

We have shown that the connected components of the finite induced subgraphs of $G$ are isomorphic to induced subgraphs of $G_0$.
It is clear that this holds for the connected components of the entire graph $G$ as well.

``$\Leftarrow$'':
Assume that $G$ is undirected and each connected component is isomorphic to an induced subgraph of $G_0$.
Then every finite induced subgraph of $G$ is a disjoint union of induced subgraphs of $G_0$ and thus belongs to $\sI \sPfsps G_0 = \Modf \qId G_0$ by Remark~\ref{rem:IPsps-closures}.
By Corollary~\ref{cor:W-ModSigma-fin}, $G \in \Mod \qId G_0$.
\end{proof}

\subsection*{Graph classes defined by forbidden subgraphs}

As an application of our results, we can show that
certain families of graphs specified by forbidden induced subgraphs
constitute graph quasivarieties.
We focus on 
the class of all graphs $H$ that have no induced subgraphs isomorphic to a strong homomorphic image of a fixed graph $G$, in other words, the class of all graphs admitting no strong homomorphism from $G$.

We start with a result that concerns an arbitrary graph $G$ and
the class of all graphs in which no induced term subgraph admits a strong homomorphism from $G$.
An induced subgraph $G$ of $H$ is called an \emph{induced term subgraph} of $H$ if $G$ is isomorphic to a term graph (see Definition~\ref{def:term-graph}).
Recall that every (strong) homomorphic image of a term graph is a term graph.

\begin{proposition}
\label{prop:forbidden-G}
Let $G$ be a graph, and let $\mathcal{K}$ be the class of all graphs $H$ such that there is no strong homomorphism of $G$ into any induced term subgraph of $H$.
Then $\mathcal{K}$ is a graph quasivariety.
\end{proposition}

\begin{proof}
We need to show that $\mathcal{K} = \Mod \qId \mathcal{K}$. The inclusion $\mathcal{K} \subseteq \Mod \qId \mathcal{K}$ is clear.
For the converse inclusion, let $H \in \Mod \qId \mathcal{K}$.
Let $T$ be an arbitrary induced term subgraph of $H$, say, without loss of generality, $T = G(t)$ for some term $t$; then $T = \reach[T]{a}$ for $a := L(t)$.
In order to show that $H \in \mathcal{K}$, it is enough to show that there is no strong homomorphism of $G$ into $T$.

Since $T$ is finite, we have $T \in \Modf \qId \mathcal{K} = \sI \sPfsps \mathcal{K}$ by Corollary~\ref{cor:W-ModSigma-fin} and Theorem~\ref{thm:char-finite}
Therefore, $T$ satisfies conditions \ref{prop:IPK-ab:a} and \ref{prop:IPK-ab:b} of Proposition~\ref{prop:IPK-ab}.
By condition \ref{prop:IPK-ab:a}, there exists a graph $G_a \in \mathcal{K}$ and a strong homomorphism $\varphi_a \colon T \to G_a$.
Then $\varphi_a(T)$ is an induced term subgraph of $G_a$.
Suppose, to the contrary, that there exists a strong homomorphism $h \colon G \to T$.
Then $\varphi_a \circ h$ is a strong homomorphism of $G$ into $\varphi_a(T)$, which is an induced term subgraph of $G_a$,
in contradiction to the definition of the class $\mathcal{K}$.
\end{proof}

The class of graphs described in Proposition~\ref{prop:forbidden-G} is a quasivariety; hence it is definable by a set of implications.
The proposition itself does not, however, provide such a defining set of implications.
We are now going to construct one.

Let $G$ be a graph.
A strongly connected component $C$ of $G$ is called a \emph{source} if there is no edge $(u,v) \in E(G)$ such that $u \in V(G) \setminus V(C)$, $v \in V(C)$.
Fix a transversal $S$ of the vertex sets of the sources of $G$.
Then $V(G) = \bigcup_{a \in S} \reach[G]{a}$, and there is no proper subset $S'$ of $S$ such that $V(G) = \bigcup_{a \in S'} \reach[G]{a}$.
For each $\varphi \colon S \to \IN$, let $S_\varphi^0 := \{a \in S \mid \varphi(a) = 0\}$, $S_\varphi^+ := \{a \in S \mid \varphi(a) > 0\}$, and let $T_\varphi$ be a directed tree comprising a root $r$ and, for each $a \in S_\varphi^+$, a directed path of length $\varphi(a)$ from $r$ to $a$, such that these directed paths are pairwise disjoint (except for the common root vertex $r$) and such that $V(T_\varphi) \cap V(G) = S_\varphi^+$.
Then $G \cup T_\varphi$ is the graph obtained by gluing $G$ and $T_\varphi$ together at the common vertices $S_\varphi^+$.
Now, let $G_\varphi$ be the graph obtained from $G \cup T_\varphi$ by gluing together $r$ and all vertices in $S_\varphi^0$, that is,
\begin{align*}
V(G_\varphi) &:= V(G \cup T_\varphi) \setminus S_\varphi^0, \\
E(G_\varphi) &:= E(G \cup T_\varphi) \cup \{(r,x) \mid \exists a \in S_\varphi^0 \colon (a,x) \in E(G)\}
\\ & \phantom{{}:= E(G \cup T_\varphi)} \cup \{(x,r) \mid \exists a \in S_\varphi^0 \colon (x,a) \in E(G)\}.
\end{align*}
Note that $G_\varphi$ is a finite graph and every one of its vertices is reachable from $r$; hence $G_\varphi$ is a term graph.

Let $\Xi_G := \{\alpha_\varphi \rightarrow \beta \mid \varphi \in \IN^S\}$, where
$\alpha_\varphi := \Sigma(G) \cup \Gamma_\mathrm{e}(T_\varphi) \cup \{x_a \approx x_r \mid a \in S_\varphi^0\}$ and
$\beta := x_r \approx \infty$.
(Recall the notation of Definition~\ref{def:Sigma-Gamma}.)

\begin{proposition}
\label{prop:forbidden-identities}
Let $G$ be a finite graph.
The class $\mathcal{K}$ of all graphs $H$ such that there is no strong homomorphism of $G$ into any induced term subgraph of $H$ is defined by the set $\Xi_G$ of implications specified above.
\end{proposition}

\begin{proof}
Assume first that $H \notin \mathcal{K}$.
Then there exists a strong homomorphism $h \colon G \to H$ such that $h(G)$ is contained in a term subgraph, say $G(t)$, of $H$.
Let $\tilde{r} := L(t)$.
Let $S$ be the transversal of the sources of $G$ that was fixed in the construction of $\Xi_G$.
For each $a \in S$, there is a directed path $p_a$ from $\tilde{r}$ to $h(a)$.
Let $\varphi \colon S \to \IN$ be the mapping that sends each $a \in S$ to the length of $p_a$.
We can clearly extend $h$ into a homomorphism $h' \colon G \cup T_\varphi \to H$ satisfying $h'(r) = \tilde{r}$.
Since $h'(a) = \tilde{r} = h'(r)$ for all $a \in S_\varphi^0$, the assignment $h'$ makes true the identities $x_a \approx x_r$ for $a \in S_\varphi^0$.
Furthermore, $h'$ makes true every identity in $\Gamma_\mathrm{e}(T_\varphi)$ and $\Sigma(G)$ by Lemma~\ref{lem:SigmaG}, parts \ref{lem:SigmaG-Gamma} and \ref{lem:SigmaG-Sigma}, respectively.
On the other hand, $h'(x_r) = \tilde{r} \neq \infty$, so $h'$ makes $\beta$ false.
Hence $H \nsatisfies \alpha_\varphi \rightarrow \beta$, so $H \notin \Mod \Xi_G$.

Assume now that $H \notin \Mod \Xi_G$.
Then there is an implication $\alpha_\varphi \rightarrow \beta \in \Xi_G$ such that $H \nsatisfies \alpha_\varphi \rightarrow \beta$, that is, there is an assignment $h \colon V(G \cup T_\varphi) \to V(H) \cup \{\infty\}$ that makes every identity in $\alpha_\varphi$ true and $\beta$ false; hence $h(x_r) \neq \infty$.
For each $a \in S_\varphi^0$, we obtain $h(x_a) = h(x_r) \neq \infty$ from the fact that $h$ makes the identity $x_a \approx x_r$ true.
By applying Lemma~\ref{lem:SigmaG}\ref{lem:SigmaG-inf} to $r$ and each vertex $a \in S_\varphi^0$, we get $h(x_b) \neq \infty$, i.e., $h(x_b) \in V(H)$ for every $b \in \reach[G \cup T_\varphi]{r}$ and for every $b \in \reach[G \cup T_\varphi]{a}$, $a \in S_\varphi^0$.
But $\reach[G \cup T_\varphi]{r} \cup \bigcup_{a \in S_\varphi^0} \reach[G \cup T_\varphi]{a} = V(G \cup T_\varphi)$, so $h(x_b) \in V(H)$ for every $b \in V(G \cup T_\varphi)$.
It follows from Lemma~\ref{lem:SigmaG}, parts \ref{lem:SigmaG-Gamma} and \ref{lem:SigmaG-Sigma}, that $h$ is a homomorphism of $G \cup T_\varphi$ into $H$ and its restriction to $V(G)$ is a strong homomorphism of $G$ into the subgraph of $H$ induced by $h(G \cup T_\varphi)$, which is a term graph because $h(G \cup T_\varphi) = h(G_\varphi)$ and $G_\varphi$ is a term graph.
Therefore $H \notin \mathcal{K}$.
\end{proof}

For an infinite graph $G$, we obtain a set of implications defining the variety of Proposition~\ref{prop:forbidden-G} by taking the union of all sets $\Xi_{G'}$ (as in Proposition~\ref{prop:forbidden-identities}) for every finite strong homomorphic image $G'$ of $G$.

Proposition~\ref{prop:forbidden-G} takes a simpler form when $G$ is a term graph.
In this case, the strong homomorphic images of $G$ into a graph $H$ are induced term subgraphs of $H$; hence, $\mathcal{K}$ is simply the class of all graphs admitting no strong homomorphism from $G$.
Furthermore, we can construct a single implication defining the class $\mathcal{K}$.

\begin{proposition}
\label{prop:ex-qv}
Let $t \in T(X)$ be a term, and
let $\mathcal{K}$ be the class of all graphs $H$ such that there does not exist any strong homomorphism of $G(t)$ into $H$.
Then $\mathcal{K}$ is a graph quasivariety, defined by the implication $\alpha \rightarrow \beta$, where $\alpha := \Sigma(G(t))$ and $\beta := L(t) \approx \infty$.
\end{proposition}

\begin{proof}
Assume $H \notin \mathcal{K}$.
Then there exists a strong homomorphism $h \colon G(t) \to H$.
By Lemma~\ref{lem:SigmaG}\ref{lem:SigmaG-Sigma}, $h$ is an assignment that makes every identity in $\Sigma(G(t)) = \alpha$ true.
Moreover, $h(L(t)) \in V(H)$, so $h(L(t)) \neq \infty$, i.e., $h$ makes $\beta$ false.
Thus $H \nsatisfies \alpha \rightarrow \beta$.

Conversely, assume $H \nsatisfies \alpha \rightarrow \beta$.
Then there exists an assignment $h \colon \var(t) \to V(H) \cup \{\infty\}$ that makes every identity in $\alpha$ true and $\beta$ false.
By Lemma~\ref{lem:SigmaG}\ref{lem:SigmaG-inf}, we have $h(x) \neq \infty$, i.e., $h(x) \in V(H)$ for all $x \in \reach[G(t)]{L(t)} = G(t)$.
It follows from Lemma~\ref{lem:SigmaG}\ref{lem:SigmaG-Sigma} that $h$ is a strong homomorphism of $G(t)$ into $H$.
Thus $H \notin \mathcal{K}$.
\end{proof}

Proposition~\ref{prop:ex-qv} in turn takes a simpler form when $G$ is a term graph that has no strong homomorphic image not isomorphic to $G$ itself.
In this case, $\mathcal{K}$ is the class of all $G$-free graphs.

\begin{example}[{cf.\ \cite[Example~4.3]{PosWes}}]
A graph $G$ is \emph{perfect} if for every induced subgraph $H$ of $G$, the chromatic number of $H$ equals the size of the largest clique in $H$.
A graph $G$ is \emph{Berge} if it has no induced subgraph isomorphic to $C_{2k+1}$ (odd cycle) or $\overline{C}_{2k+1}$ (the complement of an odd cycle) for $k \geq 2$.
Conjectured by Berge~\cite{Berge} and proved by Chudnovsky, Robertson, Seymour, and Thomas~\cite{CRST}, the ``strong perfect graph theorem'' asserts that a graph is perfect if and only if it is Berge.

The odd cycles $C_{2k+1}$ and their complements $\overline{C}_{2k+1}$ are undirected connected graphs and hence term graphs.
Consequently, the strong perfect graph theorem and Proposition~\ref{prop:ex-qv} provide a characterization of the class of perfect graphs by the following set of implications:
\[
\{
x_0 x_0 \approx \infty, \,
x_0 (x_1 x_0) \approx x_0 x_1
\}
\cup
\{ I_G \mid G \in \{C_{2k+1}, \overline{C}_{2k+1} \mid k \geq 2 \} \},
\]
where $I_G$ denotes the implication $\Sigma(G(t)) \rightarrow L(t) \approx \infty$ for some term $t$ such that $G(t) \cong G$ as in Proposition~\ref{prop:ex-qv}.
The first two identities (recall that every identity can be viewed as an implication) express that the graph is loopless and undirected.
\end{example}

\subsection*{Directed unions and the finiteness of strong pointed subproducts}

It is easy to see that the inclusion
\begin{equation}
\sI \sPsps \mathcal{K} \subseteq \Mod \qId \mathcal{K}
\label{eq:IPK-MqIK}
\end{equation}
holds for any class $\mathcal{K}$ of graphs.
For, if $G \in \sI \sPsps \mathcal{K}$,
then Proposition~\ref{prop:IPK-ab} provides strong homomorphisms from certain subgraphs of $G$ into members of $\mathcal{K}$.
By taking appropriate restrictions of these strong homomorphisms, we see that every finite induced subgraph of $G$ belongs to $\sI \sPfsps \mathcal{K}$.
Since $G$ is a directed union of all its finite induced subgraphs, we obtain $G \in \sD \sI \sPfsps \mathcal{K} = \Mod \qId \mathcal{K}$.

As shown by the following example, the inclusion \eqref{eq:IPK-MqIK} may be proper; hence directed unions cannot be omitted from the Characterization Theorem~\ref{thm:char}.

\begin{example}
\label{ex:paths}
For $n \in \IN$, denote by $P_n$ the directed path of length $n$, i.e., $V(P_n) := \{0, \dots, n\}$, $E(P_n) := \{(i,i+1) \mid 0 \leq i \leq n-1\}$, and denote by $P_\omega$ the countably infinite directed path, i.e., $V(P_n) := \IN$, $E(P_n) := \{(i,i+1) \mid i \in \IN\}$.
Let $\mathcal{K} := \{P_n \mid n \in \IN\}$ be the class of all finite directed paths.
The infinite path $P_\omega$ is the directed union of all finite paths, so $P_\omega \in \sD \sI \sPfsps \mathcal{K} = \Mod \qId \mathcal{K}$.
However, there is no strong homomorphism of $P_\omega$ into any finite path $P_n$, so $P_\omega \notin \sI \sPsps \mathcal{K}$ by Proposition~\ref{prop:IPK-ab}.
\end{example}

Nevertheless, the inclusion \eqref{eq:IPK-MqIK}
implies
\[
\sD \sI \sPfsps \mathcal{K}
\subseteq
\sD \sI \sPsps \mathcal{K}
\subseteq
\sD (\Mod \qId \mathcal{K})
=
\sD \sD \sI \sPfsps \mathcal{K}
=
\sD \sI \sPfsps \mathcal{K};
\]
hence, we can drop the finiteness condition from strong pointed subproducts in the Characterization Theorem~\ref{thm:char}, and we obtain
$\Mod \qId \mathcal{K} = \sD \sI \sPsps \mathcal{K}$.


\end{document}